\documentclass[11pt]{amsart}
\usepackage{CJK} % For Chinese characters
\usepackage{manfnt} % For cone (mancone) symbol
\usepackage{hyperref}
\usepackage{tikz}
\usetikzlibrary{arrows.meta, positioning}
\usepackage{amsfonts}
\usepackage{amsmath}
\usepackage{amssymb}
\usepackage{amscd}
\usepackage{graphicx}
\usepackage{enumitem}
\usepackage{latexsym}
\usepackage{color}
\usepackage{comment}
\usepackage{epsfig}
\usepackage{amsopn}
\usepackage{tikz-cd}
\usepackage{mathrsfs}
\usepackage{array}
\usepackage{epsfig}
\usepackage{pst-grad} % For gradients
\usepackage{pst-plot} % For axes
\usepackage[space]{grffile} % For spaces in paths
\usepackage{etoolbox} % For spaces in paths
\makeatletter % For spaces in paths
\patchcmd\Gread@eps{\@inputcheck#1 }{\@inputcheck"#1"\relax}{}{}
\makeatother

\hypersetup{pdfpagemode=UseNone}
\usepackage{booktabs}
\usepackage{longtable}
\theoremstyle{plain}
\newtheorem{lemma}{Lemma}[section]
\newtheorem*{theorem*}{Theorem}
\newtheorem*{lemma*}{Lemma}
\newtheorem*{proposition*}{Proposition}
\newtheorem*{conjecture*}{Conjecture}
\newtheorem*{corollary*}{Corollary}
\newtheorem*{problem*}{Problem}
\newtheorem{theorem}[lemma]{Theorem}

\newtheorem{corollary}[lemma]{Corollary}
\newtheorem{proposition}[lemma]{Proposition}
\newtheorem{claim}[lemma]{Claim}

\theoremstyle{definition}
\newtheorem{definition}[lemma]{Definition}
\newtheorem{example}[lemma]{Example}
\newtheorem{remark}[lemma]{Remark}

\DeclareMathOperator{\Aut}{Aut}

\DeclareMathOperator{\ord}{ord}

\DeclareMathOperator{\sHom}{\mathcal{H}\kern -.5pt\mathit{om}}
\DeclareMathOperator{\sTor}{\mathcal{T}\kern -1.5pt\mathit{or}}

\begin{document}

\date{\today}

\author[P. Vikash]{Pisya Vikash}
\address{Department of Mathematics, The Pennsylvania State University, University Park, PA 16802}
\email{pmv5172@psu.edu}


\title{Towards the classification of Poor Manifolds}

\begin{abstract}
The notion of poor manifolds was introduced by Zarhin and Bandman, who asked for their classification in \cite[Question 4]{bandman2024jordan}. In this paper, we answer this question completely in dimensions at most 3. We also classify poor compact K\"ahler manifolds of arbitrary dimension under the additional assumption that Kodaira dimension is not $-\infty$. We classify all poor $K3$ surfaces. Finally, give a sufficient condition for a compact K\"ahle manifold to have Kodaira dimension greater than or equal to zero.
\end{abstract}
\maketitle

\begingroup
  \setcounter{tocdepth}{1}
  \setlength{\parskip}{0pt}
  \tableofcontents
\endgroup

\section{Introduction}\label{sec:intro}

The notion of poor manifolds was introduced by Bandman and Zarhin
\cite{bandman2024jordan} in connection with questions on automorphism groups
and bimeromorphic geometry. We recall the definition.

\begin{definition}
A codimension-one subvariety of a complex manifold \(X\) is a closed analytic
subset \(Z\subset X\) such that
\[
\dim Z=\dim X-1,
\]
or equivalently \(\operatorname{codim}_X Z=1\).
\end{definition}

\begin{definition}
A rational curve in a complex manifold \(X\) is the image of a non-constant
holomorphic map
\[
\mathbb P^1 \longrightarrow X.
\]
\end{definition}

\begin{definition}[\cite{bandman2024jordan}]
A compact connected complex manifold \(X\) is called \emph{poor} if it contains
no closed analytic subsets of codimension \(1\) and no rational curves.
\end{definition}

Poor manifolds are interesting because the absence of divisors and rational
curves imposes strong rigidity. If \(X\) is poor, then \(X\) is
meromorphically hyperbolic in the sense of Fujiki \cite{Fujiki80}, also see
\cite{bandman2021bimeromorphic}; in particular every bimeromorphic self-map of
\(X\) is holomorphic, and hence
\[
\operatorname{Bim}(X)=\operatorname{Aut}(X).
\]
Bandman and Zarhin used this rigidity to study automorphism groups of naturally
associated manifolds, for example \(\mathbb P^1\)-bundles over poor manifolds.
At the same time, poor manifolds occur naturally: a complex torus is poor if
and only if it has algebraic dimension zero
\cite{bandman2021bimeromorphic}.

The general expectation is that poor K\"ahler manifolds should be built from
two basic types of objects: complex tori and irreducible holomorphic
symplectic manifolds. This philosophy is closely related to Campana's work on
compact K\"ahler manifolds without codimension-one subvarieties
\cite{campana2004isotrivialite,campana2006isotrivialite,campana2013compact}.
We are grateful to Fr\'ed\'eric Campana for pointing out that several
low-dimensional consequences follow from this circle of ideas. Poor manifolds are closely related to simple manifolds introduced by Campana; see for example \cite{campana1983densite,Campana2004Orbifolds,campana2006isotrivialite,Campana2011Orbifoldes,CampanaDemaillyVerbitsky2014}. Related structural results and applications to threefolds and hyperk\"ahler manifolds appear in \cite{CampanaPeternell2000,CampanaPeternell2001Kummer,CampanaHoringPeternell2016,CampanaOguisoPeternell2010}.

We recall the form of the Beauville--Bogomolov decomposition theorem used
throughout the paper.

\begin{theorem}[Beauville--Bogomolov decomposition, \cite{beauville1983varietes}]
\label{thm:theorem1}
Let \(X\) be a compact connected K\"ahler manifold with
\(c_1(X)_{\mathbb R}=0\). Then \(X\) admits a finite unramified cover
\(\widetilde X\) which decomposes holomorphically as
\[
\widetilde X
=
T\times X_1\times \cdots \times X_r
\times Y_1\times \cdots \times Y_s,
\]
where \(T\) is a complex torus, each \(X_i\) is an irreducible holomorphic
symplectic manifold, and each \(Y_j\) is a Calabi--Yau manifold.
\end{theorem}

Since points are codimension-one analytic subsets on curves, no compact
complex manifold of dimension \(1\) is poor. We therefore focus on dimension at
least \(2\). Our first main result gives the precise structural statement under the
assumption that the Kodaira dimension is not negative. It describes exactly
which finite quotients of products of poor tori and poor irreducible
holomorphic symplectic manifolds remain poor.

\begin{theorem}
\label{thm:intro-main}
Let \(X\) be a poor compact K\"ahler manifold with
\(\kappa(X)\neq -\infty\). Then there exist a poor complex torus \(T\)
(possibly a point), a product
\[
H=\prod_i H_i
\]
of poor compact irreducible holomorphic symplectic manifolds \(H_i\)
(possibly a point), and a finite group
\[
G\subset \operatorname{Aut}(T)\times \operatorname{Aut}(H)
\]
acting freely and holomorphically on \(T\times H\), such that
\[
X \cong (T\times H)/G.
\]
Moreover, the action of \(G\) preserves the factors, except possibly for
permutations of isomorphic irreducible holomorphic symplectic factors. In
particular, \(G\) fits into a short exact sequence
\[
1\longrightarrow G_T\longrightarrow G\longrightarrow G_H\longrightarrow 1,
\]
where \(G_T\subseteq \operatorname{Aut}(T)\) is finite and acts freely on
\(T\), and \(G_H\) is a finite subgroup of \(\operatorname{Aut}(H)\).
Conversely, any such free quotient \((T\times H)/G\) is poor.
\end{theorem}

\begin{proof}
See Theorem \ref{thm:main}.
\end{proof}

The assumption \(\kappa(X)\neq -\infty\) cannot presently be removed in all
dimensions. In dimension \(3\), however, the necessary uniruledness statement
is known.

\begin{theorem}[H\"oring--Peternell, \cite{horing2016minimal}]
\label{theorem:-infty}
Let \(X\) be a compact K\"ahler manifold of complex dimension \(3\). Then
\(X\) is uniruled if and only if \(\kappa(X)=-\infty\).
\end{theorem}

We shall also use the following abundance theorem.

\begin{theorem}[Campana--H\"oring--Peternell, \cite{campana2016abundance}]
Let \(X\) be a compact K\"ahler threefold. If \(K_X\) is nef, then \(K_X\) is
semiample.
\end{theorem}

Combining these results with Theorem \ref{thm:intro-main} gives the following
low-dimensional classification.

\begin{corollary}
Let \(X\) be a poor compact K\"ahler manifold of dimension \(2\) or \(3\).
Then:
\begin{itemize}
    \item if \(\dim_{\mathbb C}X=2\), then \(X\) is either a poor \(K3\)
    surface or a finite \'etale quotient of a poor complex torus;
    \item if \(\dim_{\mathbb C}X=3\), then \(X\) is a finite \'etale quotient
    of a poor complex torus.
\end{itemize}
\end{corollary}

We also record the elementary permanence property of poor manifolds that will
be used repeatedly.

\begin{lemma}[Lemma 4.2 of \cite{bandman2023simple}]
\label{lemma:lem0}
Let \(X,Y\) be compact connected complex manifolds and let
\(f:X\to Y\) be a surjective holomorphic map. Assume that \(Y\) is poor.
Then:
\begin{itemize}
    \item if every fiber \(F_y=f^{-1}(y)\) is finite, then \(X\) is poor;
    \item if every fiber \(F_y\) is poor and
    \[
    \dim F_y=\dim X-\dim Y,
    \]
    then \(X\) is poor.
\end{itemize}
In particular, a product of poor manifolds is poor.
\end{lemma}

We next turn to \(K3\) surfaces. Although the fact that a very general
\(K3\) surface is poor is classical, we give a precise period-domain
description of all poor \(K3\) surfaces. Let
\[
\mathcal P:\mathcal T_\Lambda\longrightarrow Q_\Lambda
\]
be the period map from the moduli space of marked \(K3\) surfaces to the
period domain. By the global Torelli theorem, the period point determines the
underlying \(K3\) surface up to biholomorphism, modulo the usual
inseparability phenomena. Thus, to classify poor \(K3\) surfaces, it is enough
to determine which period points correspond to poor \(K3\) surfaces.

We define a subset
\[
\mathcal U\subset Q_\Lambda
\]
consisting of the period points corresponding to poor \(K3\) surfaces. We show
that
\[
\mathcal U
=
Q_\Lambda
\setminus
\bigcup_{\substack{\delta\in\Lambda\\ \delta^2\geq -2}}
H_\delta,
\]
where
\[
H_\delta=\{[\omega]\in Q_\Lambda\mid (\omega,\delta)=0\}.
\]
In particular, \(\mathcal U\) is dense in \(Q_\Lambda\), but has empty
interior. We also study the intersection of \(\mathcal U\) with a fixed
hyperplane \(H_\delta\) when \(\delta^2<-2\), and show that poor period points
are dense in such a hyperplane but again have empty interior. Putting these results together, we obtain the following classification in
dimensions at most \(3\).

\begin{theorem}
Let \(X\) be a compact K\"ahler manifold of dimension at most \(3\). Then:
\begin{description}
    \item[\textup{(1)}] If \(\dim_{\mathbb C}X=2\), then \(X\) is poor if and
    only if one of the following holds:
    \begin{itemize}
        \item \(X\) is a \(K3\) surface such that, for any marking
        \[
        \phi:H^2(X,\mathbb Z)\longrightarrow \Lambda,
        \]
        one has
        \[
        \mathcal P([(X,\phi)])\in \mathcal U;
        \]
        \item \(X\) is a finite \'etale quotient of a complex torus of
        algebraic dimension \(0\).
    \end{itemize}

    \item[\textup{(2)}] If \(\dim_{\mathbb C}X=3\), then \(X\) is poor if and
    only if it is a finite \'etale quotient of a complex torus of algebraic
    dimension \(0\).
\end{description}
\end{theorem}
We dedicate the final section to make progress towards the classification in dimension $4$. To do that we first prove the following proposition which goes beyond K\"ahler manifolds and dimension $4$.

\begin{proposition}[Key rigidity in arbitrary dimension]
\label{prop:key-rigidity-nodiv-all-dim}
Let \(X\) be a compact complex manifold of dimension \(n\geq 3\), and assume
that \(X\) has no irreducible analytic subsets of codimension \(1\). Assume that
there exist two linearly independent holomorphic \((n-1)\)-forms
\[
\eta_1,\eta_2\in H^0(X,\Omega_X^{n-1}).
\]
Define
\[
W:=
\Bigl\{
\beta\in H^0(X,\Omega_X^1\otimes K_X):
\beta\wedge\eta_1=0
\ \text{and}\
\beta\wedge\eta_2=0
\Bigr\}.
\]
Then the following hold:
\begin{enumerate}
    \item if \(n=3\), then
    \[
    \dim_{\mathbb C} W\leq 1;
    \]
    \item if \(n\geq 4\) and
    \[
    K_X^{\otimes(n-3)}\not\simeq \mathcal O_X,
    \]
    then
    \[
    \dim_{\mathbb C} W\leq n-3.
    \]
\end{enumerate}
\end{proposition}
Using this we prove the following theorem.
\begin{theorem}\label{thm:main theorem}
    Let $X$ be a simply connected compact K\"ahler manifold of dimension $4$, assume that 
    \begin{description}
        \item[1] $X$ has no codimension one sub varieties;
        \item[2] $K_X$ is semipositive;
        \item[3] $\dim_{\mathbb{c}}H^4(TX)$ is greater or equal to $2$.
    \end{description}
    Then $\kappa(X)\geq 0$.
\end{theorem}

The theorem in sprit is related to the results in \cite{horing-lazic-lehn2025nonvanishing}. The conclusions are only different in case of $\chi(X,\mathcal{O}_X)=0$.

\section*{Notation} 
We use $\mathbb{R}$, $\mathbb{C}$, $\mathbb{Q}$, and $\mathbb{Z}$ to denote the real numbers, complex numbers, rational numbers, and integers, respectively. We denote the tangent bundle of $X$ by $TX$, the cotangent bundle by $\Omega^1_X$, and the canonical line bundle by $K_X$ (throughout the text, we deal with K\"ahler manifolds without Weil divisors, so we treat $K_X$ as a line bundle). By $h^{p,q}$ we denote the dimension of $H^{p,q}(X,\mathbb{C})$. For any compact complex manifold $X$, we denote the space of closed holomorphic one-forms by $H^0(X,\Omega^1_X)_{cl}$. We denote the automorphism group of $X$ by $\operatorname{Aut}(X)$. For an automorphism $f \in \operatorname{Aut}(X)$, we denote its fixed-point set by $\operatorname{Fix}(f)$. By \emph{very general}, we mean outside a countable union of proper subvarieties. By \emph{general}, we mean outside a finite union of proper subvarieties. We denote the Kodaira dimension of $X$ by $\kappa(X)$.

\section{General Classification}
\subsection{Preliminaries}
Throughout this section, we will be studying certain types of complex manifolds and their automorphism groups. Let us begin by recalling some key definitions.
\begin{definition} \label{definition:IHS}
    A compact K\"ahler manifold $X$ is said to be IHS (Irreducible Holomorphic Symplectic) if it is simply connected, $H^2(X,\mathcal{O}_X) \cong H^{0,2}(X)$ is one dimensional and $H^{2,0}(X)$ is generated by nowhere vanishing holomorphic symplectic $2$ form $\sigma$.
\end{definition}
This definition is motivated by the Beauville–Bogomolov decomposition theorem. Basic results about IHS manifolds can be found in \cite{huybrechts1997compact} and \cite{HuybrechtsErratum}. Note that the two-dimensional IHS manifolds are precisely $K3$ surfaces. Since $\sigma$ is a holomorphic symplectic form on $X$, it is everywhere nonzero and nondegenerate at every point of $X$. Hence, for each point $x\in X$, the bilinear pairing
\[
T_xX \times T_xX \to \mathbb{C}, \qquad (v,w)\mapsto \sigma_x(v,w)
\]
is skew-symmetric and nondegenerate. Therefore, contraction with $\sigma$ defines an isomorphism on each fiber
\[
T_xX \longrightarrow \Omega^1_{X,x}, \qquad v \longmapsto \iota_v(\sigma_x)=\sigma_x(v,\cdot).
\]
Since this construction varies holomorphically with $x$, it gives an isomorphism of holomorphic vector bundles, equivalently of locally free sheaves,
\[
TX \xrightarrow{\ \sim\ } \Omega^1_X.
\]

Passing to global sections, we obtain an induced isomorphism
\[
H^0(X,TX)\xrightarrow{\ \sim\ } H^0(X,\Omega^1_X).
\]
By assumption, $H^0(X,\Omega^1_X)=0$. It follows immediately that
\[
H^0(X,TX)=0.
\]
In other words, $X$ admits no nonzero global holomorphic vector fields. The existence of the symplectic form $\sigma \in H^0(X, \Omega^2_X)$ ensures that $\dim(X) = 2n$ is even and that the canonical sheaf $K_X = \mathcal{O}_X$ is trivial. Moreover, the algebra of holomorphic forms $\bigoplus_p H^0(X, \Omega^p_X)$ is generated by the symplectic form, yielding the following Hodge numbers:
\[
h^{p,0}(X) = 
\begin{cases} 
1 & \text{if } p \text{ is even}, \\
0 & \text{if } p \text{ is odd},
\end{cases}
\]
and consequently $\chi(\mathcal{O}_X) = n + 1$.
\begin{definition}
    An Enriques manifold is a connected complex manifold $Y$ that is not simply connected and whose universal cover $X$ is a IHS.
\end{definition}
Enriques manifolds $Y$ are compact, have even dimension $\dim(Y) = 2n$, and possess finite fundamental groups. We recall the following result from \cite{OguisoSchrer+2011+215+235}:
\begin{lemma}\label{lem:os11}
    Every Enriques Manifold is projective and the same holds for its universal cover.
    \begin{proof}
        \cite{OguisoSchrer+2011+215+235} (Corollary 2.7).
    \end{proof}
\end{lemma}
Throughout the text, our main source of codimension 1 subvarieties comes from line bundles. Therefore, we prove the following.
\begin{lemma}\label{lem:lemma1}
    Let $X$ be a compact and connected complex manifold of dimension $n$ without any codimension $1$ subvariety and $L$ be any nontrivial line bundle over $X$, then the only possible global sections of the line bundle is 0.
\end{lemma}

\begin{proof}
      Proof is obvious because, if $L$ is nontrivial and has a section, it's vanishing locus will give us a codimension $1$ subvariety. 
\end{proof}
\begin{lemma}\label{Lemma: prodp}
    Let, $X_i$ are compact complex manifolds, for $i$ inside some finite set $F$. Then if
    \[
    X = \prod_{i\in F} X_i
    \]
    is poor, then each $X_i$ is poor.
\end{lemma}

\begin{proof}
    Suppose for a contradiction that $X_{k}$ is not poor, suppose that it contains a rational curve, then $X$, will contain it. If instead it contains a codimension $1$ subvariety say $D$, then 
    \[
    D' := \prod_{i\in F \setminus k} X_i \times D
    \]
    will be a codimension 1 subvariety inside $X$, in either case we have a contradiction.
\end{proof}
The following result is due to A. Beauville.
\begin{theorem} \label{proposition : prop1}
    Let $T$ be a complex torus and $H_1, \dots, H_n$ IHS. Partition the indices by isomorphism type: for each class $[K]$, set
    \[
        I_{[K]} = \{\, i \mid H_i \simeq K \,\} \quad \text{and} \quad m_{[K]} = |I_{[K]}|.
    \]
    Then
    \[
        \operatorname{Aut}\!\left(T \times \prod_{i=1}^n H_i\right)
        \cong
        \operatorname{Aut}(T) \times \prod_{[K]} \left( \operatorname{Aut}(K)^{m_{[K]}} \rtimes {S}_{m_{[K]}} \right).
    \]
    In particular, if $H_i$ are pairwise non-isomorphic, then
    \[
        \operatorname{Aut}\!\left(T \times \prod_{i=1}^n H_i\right) \cong \operatorname{Aut}(T) \times \prod_{i=1}^n \operatorname{Aut}(H_i).
    \]
\end{theorem}

\begin{proof}
    \cite[section 3]{beauville1983some}
\end{proof}
\subsection{General classification}

In this section we give the General classification result concerning poor manifolds.
\begin{lemma}\label{lemma:aut}
Let
\[
H=\prod_{i=1}^k H_i^{n_i},
\]
where $H_i$ are pairwise non-isomorphic poor irreducible holomorphic symplectic manifolds. Let
\[
G\subset \Aut(H)
\]
be a finite subgroup. If $G$ acts freely on $H$, then $G$ is trivial.
\end{lemma}

\begin{proof}
Suppose, for a contradiction, that $G\neq \{1\}$. Choose $1\neq h\in G$. Since $G$ acts freely on $H$, we have
\[
\operatorname{Fix}(h)=\varnothing.
\]

By Theorem~\ref{proposition : prop1}, we may write
\[
h=(h_{n_1},\dots,h_{n_k}),
\qquad
h_{n_i}=(h_{n_i}^1,\dots,h_{n_i}^{n_i})\circ \sigma_i,
\]
where
\[
h_{n_i}^j\in \Aut(H_i), \qquad \sigma_i\in S_{n_i}.
\]

If every $h_{n_i}$ had a fixed point, then their product would give a fixed point of $h$ on $H$, contradicting $\operatorname{Fix}(h)=\varnothing$. Hence, there exists some $i$ such that
\[
\operatorname{Fix}(h_{n_i})=\varnothing.
\]

We claim that for some cycle
\[
c=(k_1,\dots,k_r)
\]
in the cycle decomposition of $\sigma_i$, the corresponding cycle-composition
\[
l:=h_{n_i}^{k_r}\circ \cdots \circ h_{n_i}^{k_1}\in \Aut(H_i)
\]
is fixed-point-free.

Suppose not. Then for every cycle
\[
c=(k_1,\dots,k_r)
\]
of $\sigma_i$, the automorphism
\[
h_{n_i}^{k_r}\circ \cdots \circ h_{n_i}^{k_1}
\]
has a fixed point. Fix such a cycle and choose
\[
x_{k_1}\in H_i
\]
such that
\[
h_{n_i}^{k_r}\circ \cdots \circ h_{n_i}^{k_1}(x_{k_1})=x_{k_1}.
\]
Define inductively
\[
x_{k_2}:=h_{n_i}^{k_1}(x_{k_1}),\quad
x_{k_3}:=h_{n_i}^{k_2}(x_{k_2}),\quad \dots,\quad
x_{k_r}:=h_{n_i}^{k_{r-1}}(x_{k_{r-1}}).
\]
Then
\[
h_{n_i}^{k_1}(x_{k_1})=x_{k_2},\quad
h_{n_i}^{k_2}(x_{k_2})=x_{k_3},\quad \dots,\quad
h_{n_i}^{k_r}(x_{k_r})=x_{k_1}.
\]
Thus, the tuple
\[
(x_{k_1},\dots,x_{k_r})\in H_i^r
\]
is fixed by the restriction of $h_{n_i}$ to the coordinates corresponding to the cycle $c$. Repeating this for every cycle of $\sigma_i$, we obtain a fixed point of $h_{n_i}$ on $H_i^{n_i}$, contradicting
\[
\operatorname{Fix}(h_{n_i})=\varnothing.
\]
This proves the claim.

Now, let $m=\ord(h)$. Since the cycle $c$ has length $r$, the automorphism $h^r$ acts on the $k_1$-coordinate by $l$. Hence $h^{mr}$ acts on the $k_1$-coordinate by $l^m$. But $h^{mr}=Id$, so
\[
l^m=Id.
\]
Therefore, $l$ is a finite-order fixed-point-free automorphism of $H_i$. By Lemma~\ref{lem:os11}, it follows that $H_i$ is projective. But a positive-dimensional projective manifold contains a divisor, so $H_i$ is not poor, contradicting the hypothesis. Hence $G=\{1\}$.
\end{proof}
\begin{theorem}\label{thm:main}
    Let $X$ be a poor compact K\"ahler manifold with $\kappa(X)\neq -\infty$. Then there exist a poor complex torus $T$ (possibly a point), a product $H=\prod_i H_i$ of poor compact irreducible holomorphic symplectic (IHS) manifolds (possibly a point), and a finite group $G\subset \operatorname{Aut}(T)\times\operatorname{Aut}(H)$ acting freely and holomorphically on $T\times H$ such that
    \[
    X \cong (T\times H)/G,
    \]
    and $G$ preserves the factors (and may permute isomorphic IHS factors). In particular, $G$ fits into a short exact sequence
    \[
    1 \rightarrow G_T \rightarrow G \rightarrow G_H \rightarrow 1
    \]
    where $G_T \subseteq \operatorname{Aut(T)}$ is finite acting freely holomorphically on $T$ and $G_H$ is finite subgroup of $\operatorname{Aut}(H)$. Conversely, any such free quotient $(T\times H)/G$ is poor.
\end{theorem}

\begin{proof}

\begin{description}
    \item[Step 1] Consider the Kodaira dimension $\kappa(X)$. If $\kappa(X)>0$, then there exists $m>0$ with $H^0(X,K_X^{\otimes m})\neq 0$. Any nonzero section of $K_X^{\otimes m}$ must vanish somewhere unless $K_X^{\otimes m}\simeq\mathcal O_X$; hence, by Lemma~\ref{lem:lemma1}, $X$ cannot be poor. Thus $\kappa(X)\le 0$. By assumption $\kappa(X)=0$, so $K_X^{\otimes m}\simeq\mathcal O_X$ for some $m>0$; in particular $c_1(K_X)_{\mathbb R}=0$. By Theorem \ref{thm:theorem1}, there exists a finite unramified cover $\pi:\tilde X\to X$ that decomposes as
        \[
        \tilde X \;\simeq\; T \times H_1 \times \cdots \times H_r \times Y_1 \times \cdots \times Y_s,
        \]
        as in Theorem \ref{thm:theorem1}. By Lemma \ref{lemma:lem0}, poorness is preserved under finite covers, so $\tilde X$ is poor. Observe that, $Y_i$ are projective (\cite[Chapter 7, Exercise 1]{Voisin_2002},\cite{e737bc63-03cf-3413-a9ad-63fed22470ae}). Suppose some $Y_i$ appears in the product. Then it admits an ample line bundle with a (smooth) divisor $D\subset Y_i$ of codimension $1$. Therefore, no $Y_i$ occurs due to Lemma \ref{Lemma: prodp}. Set $H = \prod_iH_i$ then $T \text{ and } H$ are poor because $\tilde{X}$ is poor. So, $X \cong (T \times H)/G$, where $G \text{ is a finite subgroup of } \operatorname{Aut(T \times H)}$.
        \item[Step 2]  By Theorem \ref{proposition : prop1} we have $G \subseteq \operatorname{Aut(T)} \times \operatorname{Aut}(H)$. Set, $G_T := \operatorname{Ker}(G \xrightarrow{projection} \operatorname{Aut}(H)) \subseteq \operatorname{Aut}(T)$ and $G_H :=\operatorname{Im}(G \xrightarrow{projection} \operatorname{Aut}(H)) \subseteq \operatorname{Aut}(H)$. Then we have a short exact sequence,
        \[
        1 \rightarrow G_T \rightarrow G \rightarrow G_H \rightarrow 1.
        \]
         Now all that is left to prove is the following claim.
         \begin{claim}
             $G_T$ acts freely on $T$.
         \end{claim}
         \begin{proof}
             Suppose for a contradiction that the action of $G_T$ on $T$ is not free. Then there exists
    \[
    g\in G_T
    \]
    such that $g$ has a fixed point. By the above short exact sequence, there exists $h\in G_H$ such that
    \[
    (g,h)\in G.
    \]
   The isomorphism
   \[
   X \cong (T\times H)/G
   \]
   is induced by a finite \'etale (equivalently, unramified) covering
   \[
   \pi:T\times H \to X.
   \]
   In particular, $G$ is the group of deck transformations of $\pi$, and therefore the action of $G$ on $T\times H$ is free. Hence $(g,h)$ has no fixed points. On the other hand,
    \[
    \operatorname{Fix}(g,h)=\operatorname{Fix}(g)\times \operatorname{Fix}(h).
    \]
    Since $\operatorname{Fix}(g)\neq \varnothing$, it follows that
    \[
    \operatorname{Fix}(h)=\varnothing.
    \]
     Therefore, from Lemma \ref{lemma:aut}, $h$ is identity. This contradicts $\operatorname{Fix}(h)=\varnothing$.
         \end{proof}
         This completes the proof of one direction. For the converse, assume that $T$ and $H$ are poor. By Lemma~\ref{lemma:lem0}, the product $T\times H$ is poor. Let
\[
\pi:T\times H \longrightarrow X:=(T\times H)/G
\]
be the quotient map. Since $G$ is finite and acts freely, $\pi$ is finite \'etale.

We claim that $X$ is poor. Suppose first that $X$ contains a codimension one analytic subset $D\subset X$. Then $\pi^{-1}(D)\subset T\times H$ is a nonempty analytic subset. Since $\pi$ is finite, its fibers are zero-dimensional, and hence every irreducible component of $\pi^{-1}(D)$ has codimension one in $T\times H$. This contradicts the fact that $T\times H$ is poor.

Now suppose that $X$ contains a rational curve $C\subset X$. Let $\nu:\mathbb{P}^1\to C$ be the normalization. Consider the fiber product
\[
Y:=(T\times H)\times_X \mathbb{P}^1.
\]
Since $\pi$ is finite \'etale, the induced map $Y\to \mathbb{P}^1$ is also finite \'etale. But $\mathbb{P}^1$ is simply connected, so every finite \'etale cover of $\mathbb{P}^1$ is trivial. Thus $Y$ is a disjoint union of copies of $\mathbb{P}^1$. In particular, $T\times H$ contains a rational curve, again a contradiction.

Therefore $X$ contains neither a codimension one analytic subset nor a rational curve. Hence $X$ is poor.
\end{description}
\end{proof}
\begin{corollary}
    Let $X$ be a poor compact K\"ahler manifold with $\kappa(X)\neq -\infty$. Assume that $X$ admits a finite \'etale cover $\widetilde X\to X$ with $\widetilde X$ simply connected. Then
    \[
    X \simeq \prod H_i
    \]
    Where each $H_i$ is a IHS manifold.
\end{corollary}

\begin{proof}
    By Theorem~\ref{thm:main}, we can write
    \[
    X \;\cong\; (T\times H)/G,
    \]
    where $T$ is a (possibly trivial) poor complex torus, $H=\prod_i H_i$ is a product of poor IHS manifolds (possibly trivial), and $G$ is a finite subgroup of $\Aut(T\times H)$ acting freely. Let $\pi:\widetilde X\to X$ be a finite \'etale cover with $\widetilde X$ simply connected. Form the fiber product
    \[
    Z := \widetilde X\times_X (T\times H).
    \]
    Then $Z\to \widetilde X$ is finite \'etale, hence $Z$ is a disjoint union of copies of $\widetilde X$, and in particular each connected component of $Z$ is simply connected. On the other hand, $Z\to T\times H$ is also finite \'etale, so each connected component of $Z$ is a finite \'etale cover of $T\times H$. Since $H$ is a product of IHS manifolds, it is simply connected; hence $\pi_1(T\times H)\cong \pi_1(T)$. A simply connected finite \'etale cover of $T\times H$ exists only if $\pi_1(T)=0$, i.e.\ $T$ is a point. Therefore $T=\{*\}$ and
    \[
    X \;\cong\; H/G.
    \]
    If $G$ is not trivial, we have a contradiction due to Lemma \ref{lemma:aut}.
\end{proof}
In \cite{bandman2021bimeromorphic} (see the discussion before Remark 1.10) it was observed that a complex torus $T$ is poor if and only if $a(T)=0$. Such examples exist in every dimension, and an explicit construction is given in \cite{bandman2023simple}. Let us give examples of such tori following \cite{bandman2023simple}. 
\begin{example}     
Let \[ f(x):=x^6+x+1 \in \mathbb{Q}[x]. \] Since $6=2g$ with $g=3$, this corresponds to the case of a purely imaginary number field of degree $2g=6$. Set \[ E:=\mathbb{Q}[x]/(f(x)). \] Let $\tilde{x}\in E$ denote the image of $x$. Choose the three roots $\alpha_1,\alpha_2,\alpha_3 \in \mathbb{C}$ of $f(x)$ with a positive imaginary part. For each $j=1,2,3$, define an embedding \[ \tau_j:E \longrightarrow \mathbb{C}, \qquad \tau_j(\tilde{x})=\alpha_j. \] Then the map \[ \Psi:E_{\mathbb{R}}:=E\otimes_{\mathbb{Q}}\mathbb{R}\longrightarrow \mathbb{C}^3, \qquad e \longmapsto (\tau_1(e),\tau_2(e),\tau_3(e)) \] is an isomorphism of $\mathbb{R}$-algebras. Now consider the lattice \[ \Gamma:=\mathbb{Z}\cdot 1 \oplus \mathbb{Z}\cdot \tilde{x} \oplus \mathbb{Z}\cdot \tilde{x}^2 \oplus \mathbb{Z}\cdot \tilde{x}^3 \oplus \mathbb{Z}\cdot \tilde{x}^4 \oplus \mathbb{Z}\cdot \tilde{x}^5 \subset E. \] Its image $\Psi(\Gamma)\subset \mathbb{C}^3$ is a discrete lattice of rank $6$, generated by \[ (1,1,1),\quad (\alpha_1,\alpha_2,\alpha_3),\quad (\alpha_1^2,\alpha_2^2,\alpha_3^2),\quad (\alpha_1^3,\alpha_2^3,\alpha_3^3),\quad (\alpha_1^4,\alpha_2^4,\alpha_3^4),\quad (\alpha_1^5,\alpha_2^5,\alpha_3^5). \] Therefore \[ T:=\mathbb{C}^3/\Psi(\Gamma) \] is a complex torus of dimension $3$. Equivalently, one may write \[ T \cong E_{\mathbb{R}}/\Gamma, \] where the complex structure of $E_{\mathbb{R}}$ is induced by the chosen embeddings $\tau_1,\tau_2,\tau_3$. This is poor because of \cite[Theorem 1.3, Proposition 2.1]{bandman2023simple}. 
\end{example}
In the next section, we study poor $K3$ surfaces. We plan to address the higher-dimensional analogue in future work.

\section{Dimension 2}
\subsection{Preliminaries}
We treat the case of dimension 2 separately due to its particularly well-understood classification. The Minimal Model Program for compact complex surfaces is firmly established by the Enriques-Kodaira classification \cite{barth2003compact} and the foundational work of Mori \cite{ab3c5836-bdbd-3fe7-9721-d43dbbc02ff0} and Miyaoka-Mori \cite{c50b404d-1559-39c7-8769-80091a236fda}. We begin by examining the minimal models for K\"ahler surfaces, as listed in \cite{barth2003compact} (page 244). Since a poor manifold has algebraic dimension zero, the only possible poor minimal surfaces are the $K3$ surfaces and the complex tori.
\begin{definition}
    A $K3$ surface is a compact complex surface $X$ with a trivial canonical bundle $K_X$ and the first Betti number $b_1(X) = 0$.
\end{definition}
The intersection form $(H^2(X, \mathbb{Z}), \cup)$ of a $K3$ surface $X$ is isomorphic to the $K3$ lattice $\Lambda = 2(-E_8) \oplus 3U$, where $U$ denotes the standard hyperbolic plane  $\left(\mathbb{Z}^2, \begin{pmatrix} 0 & 1 \\ 1 & 0 \end{pmatrix}\right)$. Consider the complex vector space $\Lambda_{\mathbb{C}} := \Lambda \otimes_{\mathbb{Z}} \mathbb{C}$ equipped with the $\mathbb{C}$-linear extension of the bilinear form $(\cdot, \cdot)$. The zero locus of the associated quadratic form in $\mathbb{P}(\Lambda_{\mathbb{C}})$ defines a smooth quadric, owing to the non-degeneracy of the bilinear form. We now define the notation for the quadric associated to $\Lambda$ as the open subset (in the classical topology):
\[
Q_{\Lambda} := \left\{ x \in \mathbb{P}(\Lambda_{\mathbb{C}}) \mid x \cdot x = 0,\ x \cdot \bar{x} > 0 \right\} \subset \mathbb{P}(\Lambda_{\mathbb{C}}),
\]
which we consider as a complex manifold, known as the period domain. 
Every $K3$ surface $X$ has a universal deformation space $\operatorname{Def}(X)$ and a universal family $f:\mathcal{X}\to \operatorname{Def}(X)$. There is a canonical way to identify $H^{2,0}(X_t)$ with $\Lambda_{\mathbb{C}}$, say given by $\phi_t$. We call the induced map 
\begin{align*}
    \mathcal{P} : \operatorname{Def}(X) &\rightarrow \mathbb{P}(\Lambda_{\mathbb{C}}),\\
    t &\rightarrow [\phi(H^{2,0}(X_t))],
\end{align*}
Period map, This map is holomorphic. For details, see \cite[Chapter 5,6]{Huybrechts_2016}.
\begin{theorem} [Local Torelli theorem]
The Local period map,
\begin{align*}
    \mathcal{P} :\operatorname{Def}(X) &\rightarrow Q_{\Lambda}
\end{align*}
is a local biholomorphism.  
\end{theorem}

\begin{proof}
    \cite{beauville1983varietes}
\end{proof}
\begin{definition}
    A marked $K3$ is a pair $(X,\phi)$, where $X$ is a $K3$ surface and $\phi$ is the lattice isomorphism $\phi : (H^2(X,\mathbb{Z}),\cup )\cong \Lambda$. We say two marked $K3$ surfaces $(X,\phi)$ and $(X',\phi')$ are equivalent $(X,\phi) \sim (X',\phi')$, if there exists a biholomorphic map $f: X \cong X'$ such that $\phi' = \phi \circ f^*$.
\end{definition}
\begin{definition}
    The moduli space of marked $K3$ surfaces is the space,
    \[
    \mathcal{T}_{\Lambda} := \{ (X,\phi) = \text{Marked K3}\}/\sim.
    \]
\end{definition}
The moduli space $\mathcal{T}_{\Lambda}$ has the structure of a $20$-dimensional complex manifold, obtained by gluing the bases of the universal deformations $\mathcal{X} \rightarrow \operatorname{Def}(X)$ for all $K3$ surfaces $X$. In particular, the $\operatorname{Def}(X)$ forms the basis of open sets in $\mathcal{T}_{\Lambda}$. The local period maps glue to the global period map,
\[
\mathcal{P}: \mathcal{T}_{\Lambda} \rightarrow Q_{\Lambda}.
\]
Now, let us recall the Global Torelli Theorem
\begin{theorem}
    With the above notation:
    \begin{enumerate}
        \item The Global period map is surjective restricted to each connected component of $\mathcal{T}_{\Lambda}$.
        \item  The Global period map is injective outside countable union of proper subvarieties in each connected component of $\mathcal{T}_{\Lambda}$.
        \item Suppose $(X, \phi), (X', \phi') \in \mathcal{T}_{\Lambda}$ are distinct inseparable points (by that we mean points that cannot be separated by open sets). Then $X \cong X'$ and $\mathcal{P}_{\Lambda}(X, \phi) = \mathcal{P}_{\Lambda}(X', \phi')$ is contained in $\alpha^\perp$ for some $0 \neq \alpha \in \Lambda$. Here, $\alpha^{\perp} = \{x\in Q_{\Lambda} | x \cdot \alpha = 0 \}$.
    \end{enumerate}
\end{theorem}

\begin{proof}
       Statement (1) was first proved by Todorov in \cite{todorov1980applications} (see also \cite{Huybrechts_2016}, Chapter 7). Statements (2) and (3) are direct consequences of technical procedures introduced by Verbitsky; detailed proofs can be found in \cite{Huybrechts_2016}, Chapter 7, Sections 4 and 5.
    \end{proof}
\subsection{Poor $K3$ Surfaces}
The following Lemma, which will be used to classify Poor $K3$ surfaces, is a direct consequence of classical theory of compact complex surfaces. We give the proof for readers convenience.
\begin{lemma}\label{LEMMA:L3.8}
    A K3 surface $X$ is poor if and only if it satisfies one of the following:
    \begin{enumerate}
        \item The Picard rank of $X$ is $0$.
        \item For all $\mathcal{O}_X \neq L \in \operatorname{Pic}(X)$, we have  $c_1(L)^2 < -2$.
    \end{enumerate}
\end{lemma}

\begin{proof}
       If $X$ is poor and $\rho(X)\neq 0$, then for any $0\neq L\in \operatorname{Pic}(X)$ the Riemann--Roch formula $\chi(X,L)=2+\frac{1}{2}c_1(L)^2$ shows that $c_1(L)^2\ge -2$ would imply $\chi(X,L)\ge 1$, hence $H^0(X,L)\neq 0$ or $H^0(X,L^{*})\neq 0$, so $L$ or $L^{*}$ is effective and yields a curve, a contradiction. Conversely, if $\rho(X)=0$ then there are no curve classes at all, and if (2) holds, then any curve $C\subset X$ would satisfy $C^2\ge -2$ by adjunction ($2p_a(C)-2=C^2$), contradicting (2).
    \end{proof}
Using the above Lemma, we would like to describe all points in $Q_{\Lambda}$ that arise as period points of a marked poor $K3$ surface. Fix a period point $[\omega] \in Q_{\Lambda}$, which determines a Hodge structure on $\Lambda_{\mathbb{C}}$:
\[
\Lambda_{\mathbb{C}} = H^{2,0}_{\omega} \oplus H^{1,1}_{\omega} \oplus H^{0,2}_{\omega}, 
\qquad H^{2,0}_{\omega} = \mathbb{C}\omega.
\]
We can then talk about integral $(1,1)$-classes, which are given by
\[
\operatorname{Pic}([\omega]) = \{\delta \in \Lambda \mid (\delta,\omega) = 0\}.
\]
Let $(X,\phi)$ be a marked $K3$ surface with $\mathcal{P}([(X,\phi)]) = [\omega]$. Then
\[
\phi(\operatorname{Pic}(X)) 
= \{\delta \in \Lambda \mid (\delta,\omega) = 0\} 
= \operatorname{Pic}([\omega]).
\]
If we pick another marked surface $(X',\phi')$ with the same period $[\omega]$ (in the same connected component as $(X,\phi)$), then by the Global Torelli theorem $X \cong X'$, and the same argument gives
\[
\operatorname{Pic}(X) \cong \operatorname{Pic}([\omega]) \cong \operatorname{Pic}(X').
\]
Thus, all $K3$ surfaces with period $[\omega]$ (up to biholomorphism) have isometric Picard lattices, and through a marking their Picard lattice is exactly the fixed sublattice
\[
\{\delta \in \Lambda \mid (\delta,\omega) = 0\}.
\]

Forget $Q_{\Lambda}$ for a moment and look at the ambient projective space $\mathbb{P}(\Lambda_{\mathbb{C}})$. Pick a nonzero vector $\delta \in \Lambda$, and consider the linear functional
\begin{align*}
    l_{\delta} : \Lambda_{\mathbb{C}} &\longrightarrow \mathbb{C} \\
    \omega &\longmapsto (\delta,\omega).
\end{align*}
Then $\mathbb{P}(\operatorname{Ker}(l_{\delta}))$ is a hyperplane in projective space. Simple linear algebra shows that
\[
H_{\delta} := Q_{\Lambda} \cap \mathbb{P}(\operatorname{Ker}(l_{\delta}))
\]
is not empty and hence a codimension $1$ subvariety of $Q_{\Lambda}$. 
\begin{remark}\label{remark:proper}
    Observe that $H_{v} = H_{w}$ if and only if $v=\lambda w$ for some $\lambda \in \mathbb{Q}$. $\operatorname{Ker}(l_v) =\operatorname{Ker}(l_w)$ implies $l_v = \lambda l_w$ and $(\cdot,\cdot)$ is nondegenerate, thus induces an injective map $V \rightarrow V^*$ and the conclusion follows.
\end{remark}

\begin{definition}
    A sublattice $P$ of $\Lambda$ is called \emph{poor} if $P=0$ or for every $0 \neq v \in P$ one has $v^2 < -2$.
\end{definition}

\begin{theorem}\label{theorem:KP}
    Let $\Lambda$ be a $K3$ lattice. By $H_{\delta}$ we denote the set 
    \[
    H_{\delta} = Q_{\Lambda} \cap \mathbb{P}(\operatorname{Ker}(l_{\delta})).
    \]
    $H_{\delta}$ is the set of period points at which $\delta$ remains of type $(1,1)$. Then we have the following.
    \begin{description}
        \item[1] The set 
    \[
    \mathcal{U} := \{[\omega] \in Q_{\Lambda} \mid \operatorname{Pic}([\omega]) \text{ is poor}\}
    \]is the complement of a countable union of closed subvarieties of codimension $1$ in $Q_{\Lambda}$. In particular,
    \[
    \mathcal{U} 
    = Q_{\Lambda} \setminus \bigcup_{\substack{\delta \in \Lambda \\ \delta^2 \geq -2}} H_{\delta}.
    \]
    \item[2] Let $(X,\phi)$ be a marked $K3$ surface. Then $X$ is poor if and only if $\mathcal{P}([(X,\phi)]) \in \mathcal{U}$.
    \end{description}
\end{theorem}

\begin{proof}
        The first statement follows from the discussion after Lemma \ref{LEMMA:L3.8}. The second statement follows from the fact that if $(X,\phi)$ is a marked $K3$ surface with $\mathcal{P}([(X,\phi)]) = [\omega]$, then $\phi(\operatorname{Pic}(X)) = \operatorname{Pic}([\omega])$, and from Lemma \ref{LEMMA:L3.8}.
\end{proof}
\begin{corollary}
We have the following description of poor $K3$ surfaces.
\begin{itemize}
    \item A very general $K3$ surface is poor.
    \item $\mathcal{U}$ has empty interior in $Q_{\Lambda}$.
\end{itemize}
\end{corollary}

\begin{proof}
        The first statement follows from Theorem \ref{theorem:KP}. On the other hand, a $K3$ surface $X$ is projective if and only if there exists a class $v \in \operatorname{Pic}(X)$ with $v^2 > 0$ (see \cite{pjateckiui1971torelli}). Thus, the period points of projective $K3$ surfaces are dense in $Q_{\Lambda}$ (\cite[Proposition 6.2.9]{Huybrechts_2016}). Since the complement of $\mathcal{U}$ contains a dense subset of projective period points, the complement is dense and, therefore, $\mathcal{U}$ cannot have a nonempty interior.
\end{proof}
Now, let us understand how the poor $K3$ surfaces sit inside $H_{\delta}$ when $\delta^{2}<-2$. This is the only case of interest for us because if $\delta^{2}\ge -2$, then no point of $H_{\delta}$ corresponds to a poor $K3$ surface (indeed, every surface that corresponds to a point in $H_{\delta}$ carries an integral $(1,1)$ class of square $\ge -2$, hence a curve due to Lemma \ref{LEMMA:L3.8}). 
\begin{remark}\label{rem:denalg}
    The period domain $Q_\Lambda$ (and similarly $H_\delta$) can be identified with the space of oriented positive definite real $2$-planes in the relevant real quadratic space $V_\mathbb{R}$; for instance, for $H_\delta$ one takes $V_\mathbb{R}=\delta^\perp_{\mathbb{R}}$, while for $Q_\Lambda$ one takes $V_\mathbb{R}=\Lambda\otimes_\mathbb{Z}\mathbb{R}$. Thus $Q_\Lambda$ is a homogeneous space for the real Lie group $SO^+(V_\mathbb{R})$, acting transitively by $g\cdot P:=g(P)$. An oriented positive $2$-plane $P\subset V_\mathbb{R}$ is called \emph{rational} if it is spanned by two vectors in
    \[
    V_\mathbb{Q}:=V_\mathbb{R}\cap(\Lambda\otimes_\mathbb{Z}\mathbb{Q}).
    \]
    Such planes exist because $V_\mathbb{Q}$ is dense in $V_\mathbb{R}$ and positivity of the restriction of the quadratic form to a $2$-plane is an open condition. Moreover, the set of rational oriented positive $2$-planes is the $SO(V_\mathbb{Q})$-orbit of any fixed rational plane. Since $SO(V_\mathbb{Q})$ is dense in $SO^+(V_\mathbb{R})$ (by weak approximation for connected $\mathbb{Q}$-groups), it follows that rational oriented positive $2$-planes are dense in the period domain. For proof details, see \cite[Proposition 5.3]{rapinchuk2014strong}.
\end{remark}
Alternatively, one could give an elementary argument as follows.
\begin{lemma}\label{lem:rational_planes_dense}
Let $V_\mathbb{R}$ be a real quadratic space and let $V_\mathbb{Q}\subset V_\mathbb{R}$ be a $\mathbb{Q}$-subspace such that $V_\mathbb{R}=V_\mathbb{Q}\otimes_\mathbb{Q}\mathbb{R}$ (equivalently, $V_\mathbb{Q}$ is dense in $V_\mathbb{R}$). Then the set of oriented positive definite real $2$-planes in $V_\mathbb{R}$ spanned by two vectors in $V_\mathbb{Q}$ is dense in the space of all oriented positive definite real $2$-planes in $V_\mathbb{R}$.
\end{lemma}

\begin{proof}
Let $P=\langle u_1,u_2\rangle_\mathbb{R}\subset V_\mathbb{R}$ be an oriented positive definite
$2$-plane. Positivity of $P$ is equivalent to the Gram matrix
\[
G(u_1,u_2)=\begin{pmatrix}
(u_1,u_1) & (u_1,u_2)\\
(u_2,u_1) & (u_2,u_2)
\end{pmatrix}
\]
being positive definite, i.e.
\[
(u_1,u_1)>0 \qquad\text{and}\qquad \det G(u_1,u_2)>0.
\]
These inequalities define an open condition on the pair $(u_1,u_2)\in V_\mathbb{R}\times V_\mathbb{R}$ because the entries of $G(u_1,u_2)$ and its determinant depend continuously on $(u_1,u_2)$. Since $V_\mathbb{Q}$ is dense in $V_\mathbb{R}$, we can choose vectors $u_1',u_2'\in V_\mathbb{Q}$ arbitrarily close to $u_1,u_2$. For $u_1',u_2'$ sufficiently close, the Gram matrix $G(u_1',u_2')$ remains positive definite, hence $P':=\langle u_1',u_2'\rangle_\mathbb{R}$ is an oriented positive definite $2$-plane. Moreover, by taking $u_1',u_2'$ close enough, the plane $P'$ can be made arbitrarily close to $P$ in the Grassmannian topology, and the orientation is preserved for sufficiently small perturbations. Thus rational oriented positive $2$-planes are dense.
\end{proof}

\begin{lemma} \label{lem: dense proj k3}
    Let $\delta\in\Lambda$ be such that $\delta^2<0$. Then
    \[
    H^{proj}_{\delta} = \{[\alpha] \in H_{\delta}| \operatorname{Pic}([\alpha]) \text{ is projective ( i.e., contains } v \text{ such that } v^2 > 0 \text{)}\}
    \]
    is dense in $H_{\delta}$.
\end{lemma}

 \begin{proof}
    Let
    \[
    V:=\delta^{\perp}_{\mathbb{R}}\subset \Lambda\otimes_{\mathbb{Z}}\mathbb{R}.
    \]
    Since $\Lambda\otimes_{\mathbb{Z}}\mathbb{R}$ has signature $(3,19)$ and $\delta^2<0$, the real quadratic space $V$ has signature $(3,18)$. Recall that the period domain $Q_{\Lambda}$ can be identified (real-analytically) with the space of oriented $2$-dimensional real subspaces on which the form is positive definite, by sending a period $[\omega]$ to the oriented plane (\cite[Chapter 6]{Huybrechts_2016})
    \[
    P_{\omega}:=\langle \operatorname{Re}(\omega),\operatorname{Im}(\omega)\rangle_{\mathbb{R}} \subset \Lambda\otimes_{\mathbb{Z}}\mathbb{R}.
    \]
    Under this identification, the condition $(\delta,\omega)=0$ is equivalent to $P_{\omega}\subset \delta^{\perp}_{\mathbb{R}}=V$. Hence we obtain an identification
    \[
    H_{\delta}\ \cong\ \bigl\{\text{oriented positive $2$-planes }P\subset V\bigr\}.
    \]
    Let
    \[
    V_{\mathbb{Q}}:=V\cap \bigl(\Lambda\otimes_{\mathbb{Z}}\mathbb{Q}\bigr).
    \]
    Call an oriented positive $2$-plane $P\subset V$ \emph{rational} if it is spanned by two vectors in $V_{\mathbb{Q}}$. The set of rational oriented positive $2$-planes is dense among all oriented positive $2$-planes in $V$ (see Remark \ref{rem:denalg}). Fix a rational oriented positive $2$-plane $P_{\mathbb{Q}}\subset V$ and let $[\omega]\in H_{\delta}$ be the corresponding period point. Choose $v,w\in V_{\mathbb{Q}}$ spanning $P_{\mathbb{Q}}$ over $\mathbb{R}$, so that $P_{\mathbb{Q}}=\langle v,w\rangle_{\mathbb{R}}$. Define
    \[
    W:=V\cap (P_{\mathbb{Q}})^{\perp} =\{x\in V\mid (x,v)=(x,w)=0\}.
    \] Since $V$ has signature $(3,18)$ and $P_{\mathbb{Q}}$ is a positive $2$-plane, its orthogonal complement $W$ has signature $(1,18)$. In particular, the cone $\{x\in W\mid x^2>0\}$ is nonempty and open in $W$. Because $v,w\in \Lambda\otimes_{\mathbb{Z}}\mathbb{Q}$, the subspace $W$ is cut out by linear equations with rational coefficients, hence it is defined over $\mathbb{Q}$:
    \[
    W_{\mathbb{Q}}:=W\cap \bigl(\Lambda\otimes_{\mathbb{Z}}\mathbb{Q}\bigr) =\ker\Bigl(F:\Lambda\otimes_{\mathbb{Z}}\mathbb{Q}\longrightarrow \mathbb{Q}^{3}\Bigr),
    \]
    where
    \[ 
    F(x)=\bigl((x,\delta),(x,v),(x,w)\bigr).
    \]
    Therefore, $W=W_{\mathbb{Q}}\otimes_{\mathbb{Q}}\mathbb{R}$, and in particular $W_{\mathbb{Q}}$ is dense in $W$. Choose $x\in W_{\mathbb{Q}}$ with $x^2>0$. Clearing the denominators, there exists an integer $N\ge 1$ such that
    \[
    L:=Nx \in W\cap \Lambda \qquad\text{and}\qquad L^{2}>0.
    \]
    Now $L\in W\subset (P_{\mathbb{Q}})^{\perp}$ implies $(L,\omega)=0$, and hence $L\in \operatorname{Pic}([\omega])$. Also, $P_{\mathbb{Q}}\subset V=\delta^{\perp}_{\mathbb{R}}$ implies $(\delta,\omega)=0$, hence $\delta\in \operatorname{Pic}([\omega])$. By the projectivity criterion for K3 surfaces, the existence of a class $L\in \operatorname{Pic}([\omega])$ with $L^{2}>0$ implies that the corresponding K3 surface is projective, i.e. \ $[\omega]\in H_{\delta}^{\mathrm{proj}}$. The density of rational oriented positive $2$-planes concludes the proof (Again by Remark \ref{rem:denalg} ).
    \end{proof}
\begin{proposition}
    Let $\delta \in \Lambda$ be a primitive vector such that $\delta^2 < -2$. Then the set
    \[
        H^{\mathrm{poor}}_{\delta} := \{[\alpha] \in H_{\delta} \mid \operatorname{Pic}([\alpha]) \text{ is poor} \}
    \]
    is dense in $H_{\delta}$ and has empty interior. Therefore, there exists a poor $K3$ surface $X$ with $\operatorname{Pic}(X) \neq 0$.
\end{proposition}

\begin{proof}
        By Theorem~\ref{theorem:KP}, we have
        \[
            H^{\mathrm{poor}}_{\delta} 
            = \mathcal{U} \cap H_{\delta} 
            = H_{\delta} \setminus 
            \bigcup_{\substack{\gamma \in \Lambda \\ \gamma^2 \geq -2 \\ a\gamma \neq b \delta \  \forall a,b \in \mathbb{Z}\setminus \{0\}}} 
            \bigl(H_{\delta} \cap H_{\gamma}\bigr).
        \]
        Each $H_{\delta} \cap H_{\gamma}$ is either empty or a proper analytic subset of the complex manifold $H_{\delta}$, and there are only countably many such $\gamma$. If it is not empty and not proper, then 
        \[
        H_{\delta} = H_{\gamma}.
        \]
        By Remark \ref{remark:proper}, we have 
        \[
        \delta = \lambda\gamma, \qquad 0\neq \lambda \in \mathbb{Q}.
        \]
        Clearing the denominators, we obtain a contradiction to the fact that $ a\gamma \neq b \delta \  \forall a,b \in \mathbb{Z}\setminus \{0\}$. Therefore, the complement of their union, $H^{\mathrm{poor}}_{\delta}$, is a dense $G_\delta$ subset of $H_{\delta}$. It has an empty interior due to the Lemma \ref{lem: dense proj k3}.
\end{proof}
\begin{remark}
    Define the \emph{moduli space of marked poor $K3$ surfaces} by
    \[
        \mathcal{M}^{\mathrm{poor}} := \mathcal{P}^{-1}(\mathcal{U}) \subset \mathcal{T}_{\Lambda}.
    \]
    By the global Torelli theorem and the description of $\mathcal{U}$, the points of $\mathcal{M}^{\mathrm{poor}}$ are exactly the marked $K3$ surfaces $(X,\phi)$ such that $\operatorname{Pic}(X)$ is poor. Moreover, $\mathcal{M}^{\mathrm{poor}}$ is a dense $G_\delta$-subset of $\mathcal{T}_{\Lambda}$ whose complement is a countable union of codimension~$1$ analytic subvarieties. After forgetting the marking, the image of $\mathcal{M}^{\mathrm{poor}}$ in the moduli stack of $K3$ surfaces defines a substack, which we refer to as the moduli stack of (unmarked) poor $K3$ surfaces.
\end{remark}
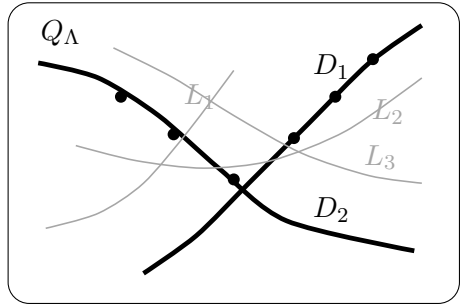
\begin{figure}[ht]
\centering
\begin{tikzpicture}[
    thickcurve/.style={line width=1.8pt},
    thinline/.style={gray!70, line width=0.6pt},
    dot/.style={circle, fill=black, inner sep=1.6pt}
]

% Ambient space X (size 6 x 4)
\draw[rounded corners=8pt] (0,0) rectangle (6,4);
\node at (0.7,3.6) {$Q_{\Lambda}$};

% -------------------------
% Thick curve 1 (unlabeled)
% -------------------------
\draw[thickcurve]
  plot[smooth] coordinates {
    (0.4,3.2) (1.2,3.0) (2.0,2.5) (2.8,1.8) (3.7,1.1) (5.4,0.7)
  };
  \node at (4.3,1.25) {$D_2$};
% Dots on thick curve 1 (unlabeled)
\node[dot] at (1.5,2.75) {};
\node[dot] at (2.2,2.25) {};
\node[dot] at (3.0,1.65) {};

% -------------------------
% Thick curve 2 (unlabeled)
% -------------------------
\draw[thickcurve]
  plot[smooth] coordinates {
    (1.8,0.4) (2.5,0.9) (3.2,1.6) (4.0,2.4) (4.8,3.2) (5.5,3.7)
  };
  \node at (4.3,3.15) {$D_1$};
% Dots on thick curve 2 (unlabeled)
\node[dot] at (3.8,2.20) {};
\node[dot] at (4.35,2.75) {};
\node[dot] at (4.85,3.25) {};

% -------------------------
% Three thin curves (labeled)
% -------------------------
\draw[thinline]
  plot[smooth] coordinates {
    (0.5,1.0) (1.2,1.2) (1.8,1.6) (2.4,2.3) (3.0,3.1)
  };
\node[gray!70] at (2.55,2.75) {$L_1$};

\draw[thinline]
  plot[smooth] coordinates {
    (0.9,2.1) (1.7,1.9) (2.6,1.8) (3.5,1.9) (4.5,2.3) (5.5,3.0)
  };
\node[gray!70] at (5.05,2.55) {$L_2$};

\draw[thinline]
  plot[smooth] coordinates {
    (1.4,3.4) (2.2,3.0) (3.0,2.5) (3.9,2.0) (4.8,1.7) (5.5,1.6)
  };
\node[gray!70] at (4.95,1.95) {$L_3$};

\end{tikzpicture}
\caption{Moduli space of poor $K3$ surfaces.}
\end{figure}
Combining these two results, we have Figure 1. As we can see, we have a wall and chamber (not open) decomposition of $Q_{\Lambda}$. where every point in the chamber represents a period point of poor $K3$ surface, in fact they have no line bundles at all. Not every wall contains poor $K3$ surfaces, but if a wall contains a poor $K3$ surface, then the set of all period points of poor $K3$ surfaces is dense, with an empty interior. The wall that contains a poor $K3$ surface is governed by $\delta^2$. We can think of each $L_{i}$ as $H_{\delta}$ for some $\delta^2 \geq -2$, and $D_i$ as $H_{\delta}$ for some $\delta^2 < -2$. The dotted points are the boundary points of the poor period points.
Finally we can put all of this as one Theorem.
\begin{theorem}
    Let $X$ be a compact K\"ahler manifold of dimension $\leq 3$.
    \begin{description}
        \item[1] If dimension $X$ is $2$, then $X$ is poor if and only if one of the following holds
        \begin{itemize}
            \item $X$ is a $K3$ surface such that, given any marking, $\phi : H^2(X,\mathbb{Z}) \rightarrow \Lambda$, we have $\mathcal{P}([(X,\phi)]) \in \mathcal{U}$.
            \item $X$ is a torus of algebraic dimension $0$.
        \end{itemize} 
        \item[2] If dimension $X$ is $3$, then $X$ is poor if and only if $X$ is a torus of algebraic dimension $0$.
    \end{description}
\end{theorem}
\begin{proof}
    The first statement follows from Theorem \ref{thm:main}, Theorem \ref{theorem:KP}, and \cite{bandman2021bimeromorphic}. The second statement follows from Theorem \ref{thm:main} and \cite{bandman2021bimeromorphic}.
\end{proof}

\section{Higher dimensional Poor manifolds}
\subsection{Preliminaries}
\begin{lemma}[Hartogs extension for vector bundles]
\label{lem:hartogs-vector-bundle}
Let \(X\) be a complex manifold and let \(Z\subset X\) be an analytic subset of
codimension at least \(2\). Set \(U:=X\setminus Z\). If \(E\) is a holomorphic
vector bundle on \(X\), then restriction gives an isomorphism
\[
H^0(X,E)\xrightarrow{\ \sim\ } H^0(U,E|_U).
\]
\end{lemma}

\begin{proof}
The statement is local on \(X\). After trivializing \(E\) on a coordinate open
set, a section of \(E|_U\) is a tuple of holomorphic functions on the complement
of an analytic subset of codimension at least \(2\). By Hartogs' extension
theorem, each of these functions extends uniquely. These local extensions glue
by uniqueness.
\end{proof}

\begin{lemma}[Meromorphic functions are constant]
\label{lem:meromorphic-functions-constant}
Let \(X\) be a compact connected complex manifold with no irreducible analytic
subsets of codimension \(1\). Then every meromorphic function on \(X\) is
constant.
\end{lemma}

\begin{proof}
Suppose \(f\) is a nonconstant meromorphic function on \(X\). Then \(f\) defines
a nonconstant meromorphic map
\[
X\dashrightarrow \mathbb P^1.
\]
For a general point \(a\in \mathbb P^1\), the fiber \(f^{-1}(a)\) is a nonempty
analytic subset of codimension \(1\). This gives a divisor on \(X\), contradicting
the assumption. Hence \(f\) is constant.
\end{proof}

\begin{lemma}[Generic proportionality implies constant proportionality]
\label{lem:generic-proportionality-constant}
Let \(X\) be a compact connected complex manifold with no irreducible analytic
subsets of codimension \(1\). Let \(E\) be a holomorphic vector bundle on \(X\),
and let
\[
s,t\in H^0(X,E)
\]
be two global sections. Suppose \(t\neq 0\) and \(s,t\) are generically
proportional. Then there exists a constant \(c\in\mathbb C\) such that
\[
s=ct.
\]
\end{lemma}

\begin{proof}
On the dense open set where \(t\neq 0\) and the proportionality is defined,
there is a meromorphic function \(f\) such that
\[
s=ft.
\]
By Lemma~\ref{lem:meromorphic-functions-constant}, \(f\) is constant. Thus
\(s=ct\) on a dense open subset of \(X\), and hence everywhere by holomorphicity.
\end{proof}

\begin{lemma}[Generic linear dependence gives a constant linear relation]
\label{lem:generic-linear-dependence-constant}
Let \(X\) be a compact connected complex manifold with no irreducible analytic
subsets of codimension \(1\). Let \(E\) be a holomorphic vector bundle on \(X\),
and let
\[
s_1,\dots,s_m\in H^0(X,E)
\]
be global sections. If \(s_1,\dots,s_m\) are generically linearly dependent,
then they satisfy a nontrivial constant linear relation:
\[
c_1s_1+\cdots+c_ms_m=0
\]
for some \((c_1,\dots,c_m)\in\mathbb C^m\setminus\{0\}\).
\end{lemma}

\begin{proof}
Choose a maximal generically linearly independent subcollection. After relabeling,
we may assume this subcollection is
\[
s_1,\dots,s_k
\]
with \(k<m\). Then \(s_{k+1}\) is generically a linear combination of
\(s_1,\dots,s_k\). Hence there exist meromorphic functions
\(f_1,\dots,f_k\) on \(X\) such that
\[
s_{k+1}=f_1s_1+\cdots+f_ks_k
\]
on a dense open subset of \(X\). By
Lemma~\ref{lem:meromorphic-functions-constant}, each \(f_i\) is constant.
Therefore
\[
s_{k+1}-f_1s_1-\cdots-f_ks_k=0
\]
is a nontrivial constant linear relation among the sections.
\end{proof}

A holomorphic line bundle \(L\) on a compact K\"ahler manifold \(X\) is called
\emph{semipositive} if there exists a smooth Hermitian metric \(h\) on \(L\)
such that
\[
\sqrt{-1}\,\Theta_h(L)\geq 0
\]
as a smooth real \((1,1)\)-form.

In \cite[Theorem 2.7.3]{demailly2001pseudo}, Demailly--Peternell--Schneider prove the following result. 
\begin{theorem}\label{thm:use}
    If $X$ is a compact K\"ahler manifold of dimension $n$, with semipositive canonical bundle then at least one of the following holds.
    \begin{itemize}
        \item The algebraic dimension $a(X)>0$.
        \item The Euler characteristic $\chi(X,\mathcal{O}_X)=0$. 
    \end{itemize}
\end{theorem}
\subsection{Main Theorem}
We now prove Proposition \ref{prop:key-rigidity-nodiv-all-dim}.

\begin{proof}
On any complex \(n\)-fold there is a canonical isomorphism \(T_X\otimes K_X\simeq \Omega_X^{n-1}\). Let \(w_i\in H^0(X,T_X\otimes K_X)\) be the section corresponding to \(\eta_i\). Equivalently, \(w_i\) gives a bundle map \(K_X^{-1}\to T_X\). By assumption, the zero loci of \(w_1\) and \(w_2\) have codimension at least \(2\). We first claim that \(w_1\wedge w_2\) is not identically zero as a section of \(\wedge^2T_X\otimes K_X^{\otimes 2}\). Indeed, if \(w_1\wedge w_2\equiv 0\), then \(w_1\) and \(w_2\) are generically proportional. By Lemma~\ref{lem:generic-proportionality-constant}, they are constant multiples of each other, and hence \(\eta_1\) and \(\eta_2\) are linearly dependent, a contradiction.

Thus \(w_1\wedge w_2\) is a nonzero global section of \(\wedge^2T_X\otimes K_X^{\otimes 2}\), and again its zero locus has codimension at least \(2\). Let \(U\subset X\) be the complement of the union of the zero loci of \(w_1\), \(w_2\), and \(w_1\wedge w_2\). Then \(\operatorname{codim}(X\setminus U)\geq 2\). On \(U\), the images of \(w_1\) and \(w_2\) define line subbundles \(L_i:=\operatorname{im}(w_i)\subset T_U\), and \(L_i\simeq K_U^{-1}\). Since \(w_1\wedge w_2\) is nowhere vanishing on \(U\), these line subbundles are pointwise linearly independent. Hence they span a rank \(2\) subbundle \(D:=L_1\oplus L_2\subset T_U\).

Let \(A:=\operatorname{Ann}(D)\subset \Omega_U^1\). Then \(A\) has rank \(n-2\), and there is an exact sequence
\[
0\longrightarrow A\longrightarrow \Omega_U^1\longrightarrow D^*\longrightarrow 0.
\]
Let \(\beta\in W\), and view \(\beta\) as a bundle map \(\beta^\#:T_X\to K_X\). The condition \(\beta\wedge\eta_i=0\) is equivalent to saying that \(\beta^\#\) vanishes on \(L_i\). Indeed, locally write \(\beta=\theta\otimes\Omega\) and \(\eta_i=\iota_{v_i}\Omega\), where \(\Omega\) is a local generator of \(K_X\). Then
\[
\beta\wedge\eta_i=(\theta\otimes\Omega)\wedge\iota_{v_i}\Omega
=\theta(v_i)\,\Omega\otimes\Omega,
\]
up to the usual harmless sign. Hence \(\beta\wedge\eta_i=0\) if and only if \(\theta(v_i)=0\), which is exactly the condition that \(\beta^\#\) vanish on \(L_i\). Therefore every \(\beta\in W\) restricts on \(U\) to a section of \(A\otimes K_U\), and we get an injection
\[
W\hookrightarrow H^0(U,A\otimes K_U).
\]

Set \(E:=A\otimes K_U\). Then \(\operatorname{rk}E=n-2\). We compute its determinant. From the exact sequence above,
\[
\det A\simeq \det(\Omega_U^1)\otimes \det(D^*)^{-1}
\simeq K_U\otimes \det D.
\]
Since \(D=L_1\oplus L_2\) and \(L_i\simeq K_U^{-1}\), we have \(\det D\simeq K_U^{-2}\). Hence \(\det A\simeq K_U^{-1}\), and therefore
\[
\det E=\det(A\otimes K_U)=\det A\otimes K_U^{\otimes(n-2)}
\simeq K_U^{\otimes(n-3)}.
\]

First suppose \(n\geq 4\). If \(\dim_{\mathbb C}W\geq n-2\), choose linearly independent elements \(\beta_1,\dots,\beta_{n-2}\in W\). Their restrictions to \(U\) are sections of \(E\), so
\[
s:=\beta_1\wedge\cdots\wedge\beta_{n-2}
\in H^0(U,\det E)\simeq H^0\bigl(U,K_U^{\otimes(n-3)}\bigr).
\]
We claim that \(s\not\equiv 0\). If \(s\equiv 0\), then \(\beta_1,\dots,\beta_{n-2}\) are generically linearly dependent as sections of \(E\). By Lemma~\ref{lem:generic-linear-dependence-constant}, they satisfy a nontrivial constant linear relation, contradicting their linear independence in \(W\). Thus \(s\neq 0\). By Lemma~\ref{lem:hartogs-vector-bundle}, \(s\) extends uniquely across \(X\setminus U\) to a nonzero section of \(K_X^{\otimes(n-3)}\). Since \(X\) has no divisors, this section is nowhere vanishing. Hence \(K_X^{\otimes(n-3)}\simeq \mathcal O_X\). Therefore, if \(K_X^{\otimes(n-3)}\not\simeq \mathcal O_X\), we must have \(\dim_{\mathbb C}W\leq n-3\).

It remains to treat the case \(n=3\). Then \(A\) has rank \(1\), and the same determinant computation gives \(A\otimes K_U\simeq \mathcal O_U\). Hence
\[
W\hookrightarrow H^0(U,\mathcal O_U).
\]
By Lemma~\ref{lem:hartogs-vector-bundle}, \(H^0(U,\mathcal O_U)\simeq H^0(X,\mathcal O_X)\). Since \(X\) is compact and connected, \(H^0(X,\mathcal O_X)=\mathbb C\). Thus \(\dim_{\mathbb C}W\leq 1\), proving the lemma.
\end{proof}

We now prove Theorem \ref{thm:main theorem}.

\begin{proof}
Assume, for contradiction, that \(\kappa(X)=-\infty\). Since \(K_X\) is
semipositive, Theorem~\ref{thm:use} gives either \(a(X)>0\) or
\(\chi(X,\mathcal O_X)=0\). The first case is impossible because \(X\) has no
codimension-one subvarieties, hence no nonconstant meromorphic functions. Thus
\[
\chi(X,\mathcal O_X)=0.
\]

Since \(X\) is simply connected and K\"ahler, \(h^{1,0}(X)=0\). Also
\(\kappa(X)=-\infty\) gives \(h^0(X,K_X)=0\). Therefore
\[
0=\chi(X,\mathcal O_X)
=1+h^{2,0}(X)-h^{3,0}(X),
\]
so
\[
h^{3,0}(X)=1+h^{2,0}(X).
\]
Moreover \(h^{2,0}(X)\neq 0\), because otherwise \(X\) would be projective by
Kodaira's criterion, hence would contain a divisor. Thus
\[
h^{3,0}(X)\geq 2.
\]

Choose linearly independent
\[
\eta_1,\eta_2\in H^0(X,\Omega_X^3).
\]
Since \(\kappa(X)=-\infty\), we have \(H^0(X,K_X^2)=0\). Hence for every
\(\beta\in H^0(X,\Omega_X^1\otimes K_X)\),
\[
\beta\wedge\eta_i\in H^0(X,K_X^2)=0.
\]
Therefore
\[
W=H^0(X,\Omega_X^1\otimes K_X).
\]
By Serre duality,
\[
H^4(X,T_X)^\vee\simeq H^0(X,\Omega_X^1\otimes K_X),
\]
so the hypothesis gives
\[
\dim_{\mathbb C}W\geq 2.
\]

But Proposition~\ref{prop:key-rigidity-nodiv-all-dim}, applied with \(n=4\),
gives
\[
\dim_{\mathbb C}W\leq 1
\]
because \(K_X\not\simeq\mathcal O_X\) under the assumption
\(\kappa(X)=-\infty\). This contradiction proves \(\kappa(X)\geq 0\).
\end{proof}
\section*{Acknowledgments} I am grateful to my advisors, John Lesieutre, Yuriy Zarhin, and Tatiana Bandman, for useful and stimulating discussions, as well as for their very helpful comments. I am also grateful to Frédéric Campana for helpful comments on the first version of this paper. I would also like to thank several graduate students at Penn State, including Andy B. Day, Eugene Henninger-Voss, Neelarnab Raha, Satwata Hans, Xingkai Wang, and Louis Diaz, for helpful discussions.

\bibliographystyle{amsalpha}

\end{document}